\documentclass[12pt]{article}
\usepackage{amsmath}
\usepackage{amsthm,amscd,amssymb}
\usepackage{graphicx,psfrag,epsf}
\usepackage{enumerate}
\usepackage{natbib}
\usepackage{url} 
\usepackage{mathrsfs}
\usepackage{setspace}
\newcommand{\blind}{0}

\newtheorem{defi}{Definition}
\newtheorem{lemma}[defi]{Lemma}
\newtheorem{remark}[defi]{Remark}

\newtheorem{theorem}[defi]{Theorem}

\newtheorem{example}[defi]{Example}

\newcommand{\E}{{\mathbb{E}}}
\newcommand{\1}{\mbox{$\mathbb{I}$}}
\newcommand{\lebesgue}{\ensuremath{\lambda\!\!\!\;\!\lambda}}

\begin{document}

\def\spacingset#1{\renewcommand{\baselinestretch}%
{#1}\small\normalsize} \spacingset{1}


\if0\blind
{
  \title{\bf The false discovery rate (FDR) of multiple tests in a class room lecture}
  \author{Benditkis Julia\thanks{
    The authors gratefully acknowledge \textit{to SAW-Project "Multiplizität, Modellvalidierung und Reproduzierbarkeit in hochdimensionalen Microarray-Daten" and to Deutsche Forschungsgemeinschaft (DFG)}}\hspace{.2cm}\\
    Heinrich-Heine University D\" usseldorf,\\
    Heesen Philipp\\
    Heinrich-Heine University D\" usseldorf\\
    and \\
   Janssen Arnold \\
    Heinrich-Heine University D\" usseldorf.}
  \maketitle
} \fi

\if1\blind
{
  \bigskip
  \bigskip
  \bigskip
  \begin{center}
    {\LARGE\bf The False Discovery Rate (FDR) of Multiple Tests in a Class Room Lecture.}
\end{center}
  \medskip
} \fi

\bigskip
\begin{abstract}
Multiple tests are designed to test a whole collection of null hypotheses simultaneously. Their quality is often judged by the false discovery rate (FDR), i.e. the expectation of the quotient of the number of false rejections divided by the amount of all rejections. The widely cited Benjamini and Hochberg (BH) step up multiple test controls the FDR under various regularity assumptions. In this note we present a rapid approach to the BH step up and step down tests. Also sharp FDR inequalities are discussed for dependent p-values and examples and counter-examples are considered. In particular, the Bonferroni bound is sharp under dependence for control of the family-wise error rate.
\end{abstract}

\noindent%
{\it Keywords:}  False Discovery Rate (FDR), multiple testing, Benjamini and Hochberg Theorem, step up test, step down test.\newline

MSC classes: \	{\bf 62G10}\newline

\vfill

\newpage
\spacingset{1.45} 
\doublespacing

\section{Introduction}
The pioneer step up multiple test of Benjamini and Hochberg (1995) is a common and widely applied tool to bound the FDR. Up to now it is cited more than 28000 times. For a survey about multiple testing we refer to Pigeot (2000) and Dudoit and van der Laan (2008) who also sketched applications for instance in the analysis of genome data. Below a very rapid approach to Benjamini and Hochberg type tests is presented. The technical details are mostly special cases of the recent work of Heesen and Janssen (2015, 2016), Benditkis (2015) and others, where additional literature is discussed. We learned from the previous work of Benjamini and Yekutieli (2001), Finner and Roters (2001), Sarkar (2002) and Blanchard et al. (2014). A multiple testing problem consists of $m$ null hypotheses 
$
(H_1,p_1), ..., (H_m,p_m)
$
with associated p-values $p_i, \ i=1,...,m.$ Assume that all p-values arise from the same experiment given by one data set, where each $p_i$ can be used for testing the traditional null $H_i$. The p-value vector
$
p=(p_1,...,p_m)\in[0,1]^m
$
is a random variable based on an unknown distribution $P.$ Recall that simultaneous inference can be established by so called multiple tests $\phi=\phi(p),$
$
\phi=(\phi_1,...,\phi_m):[0,1]^m\rightarrow \{0,1\}^m,
$
which rejects the null $H_i$ iff $\phi_i(p)=1$.\newline
Below we are mainly concerned with the well established basic independence (BI) assumptions (BI) (1)-(3).
\begin{itemize}
\item[(BI)(1)] The set of hypotheses can be divided in the disjoint union $I_0\bigcup I_1=\{1,...,m\}$ of unknown portions of true null $I_0$ and false null $I_1,$ respectively.
\item[(BI)(2)] The vectors of p-values $\left(p_i\right)_{i\in I_0}$ and $\left(p_i\right)_{i\in I_1}$ are independent, where each dependence structure is allowed for the ``false" p-values within $\left(p_i\right)_{i\in I_1}.$
\item[(BI)(3)] The p-values $\left(p_i\right)_{i\in I_0}$ of true null hypotheses are independent and stochastically larger (or equal) compared to the uniform distribution on $[0,1],$ i.e., $P(p_i\leqslant x)\leqslant x$ for all $x\in[0,1]$ and $i\in I_0.$
\end{itemize}
If in addition to (BI)(3) the p-values $(p_i)_{i \in I_0}$ are i.i.d. uniformly distributed on $[0,1]$ then we call it BI model with uniform true p-values. \newline
Observe that (BI)(3) allows p-values related to conservative single tests for $H_i, i\in I_0.$\newline
The amount $m_0:=\# I_0$ ($m_1:=\# I_1=m-m_0$) of true null (false null) is fixed but unknown. Given a multiple test $\phi$ the integers 
\begin{align*}
R=R(p):=\#\{i: \ \phi_i(p)=1\}, \ V=V(p):=\#\{i\in I_0 : \ \phi_i(p)=1\}
\end{align*}
count the numbers of all rejected null and falsely rejected null, respectively.\newline
It is well known that the control of the type 1 error probability of a multiple test by the family-wise error rate $\text{FWER}:=P\left(V>0\right)$ is often much too restrictive. To overcome this difficulty the false discovery rate (FDR) was developed as an error control criterion in the 90's. It is defined as the expectation of the false discovery proportion
\begin{align}\label{deffdr}
\text{FDR}:=\E\left[\frac{V}{\max(R,1)}\right]=\sum\limits_{i\in I_0}\E\left[\frac{\phi_i}{\max(R,1)}\right].
\end{align}
It is the aim to calculate the FDR of a multiple test and to control the FDR by a pre-specified level $\alpha\in(0,1),$ i.e. $\text{FDR}\leqslant \alpha.$
\section{Results for step up (SU) tests} 
Consider the order statistics  $ p_{1:m}\leqslant p_{2:m}\leqslant \ldots \leqslant p_{m:m}$ of the underlying p-value vector $p$. They are compared with given critical values $0=\alpha_0<\alpha_1\leqslant \alpha_2\leqslant \ldots \leqslant \alpha_m<1.$ The appertaining step up multiple test rejects $H_i$ whenever $p_i\leqslant \alpha_R.$ Thereby, $R$ is the number of rejections defined by
\begin{align}\label{defSU}
R:=\max\{j: \ p_{j:m}\leqslant \alpha_j\},
\end{align}
and $\max\{\emptyset\}=0.$
The famous so called Benjamini and Hochberg SU test is given by linear critical values $\alpha_i=b_i$ with
\begin{align}\label{BHcv}
b_i:=\frac{i\alpha}{m}, \ i=0,...,m.
\end{align}
Under the BI assumptions the next Theorem, called BH Theorem, yields its FDR control, see Benjamini and Hochberg (1995), Benjamini  and Yekutieli (2001) and Finner and Roters (2001). We present a rapid proof which summarizes and focuses various arguments given in the literature. 
\begin{theorem}[\textit{\cite{b_h_95}}]
Consider the BH step up test given by critical values (\ref{BHcv}). Under the basic independence assumptions (BI)(1)-(3) we have the inequality $\text{FDR}\leqslant \frac{m_0}{m}\alpha$ and equality under the basic independence model with uniform true p-values.
\end{theorem}
\begin{proof}(along \cite{h_j_15.1}):
Introduce for fixed $i\in I_0$ the vector $p^{(i)}=(p_1,...,0,...,p_{m})$ which coincides up to the position $i$ with $p$ and $p_i$ is replaced by $0.$ If $H_i$ is rejected its p-value only obeys the constructions $p_i\leqslant \alpha_{R(p)}=\frac{R(p)}{m}\alpha.$ On the set $\{p_i\leqslant \alpha_{R(p)}\}$ we may therefore substitute $p$ by $p^{(i)}$ and $R(p)=R(p^{(i)})$ holds. In addition we have in all cases $P(p_i\in(\alpha_{R(p)},\alpha_{R(p^{(i)})}])=0$ for an SU test. Since $p_i$ and $p^{(i)}$ are independent we may apply Fubini's Theorem
\begin{align}\label{firsteqfdr}
\E\left[\frac{\1(p_i\leqslant \alpha_{R(p)})}{\max(R(p),1)}\right]&=\E\left[\frac{\1(p_i\leqslant \alpha_{R(p^{(i)})})}{R(p^{(i)})}\right]\\
&\leqslant \E\left[\frac{\alpha_{R(p^{(i)})}}{R(p^{(i)})}\right]=\frac{\alpha}{m},
\end{align}
where first the integration is done via $p_i$ and equality holds for uniformly distributed $p_i.$ Definition (\ref{deffdr}) implies the result.
\end{proof}
The presented method of proof has further applications which applies to least favorable configurations of ``false p-values" for the FDR.
\begin{theorem}[Benjamini and Yekutieli (2001)]\label{beyek}
Consider a SU multiple test based on critical values $\left(\alpha_i\right)_{i\leqslant m}.$ Under the basic independence assumptions with uniform true p-values the following assertions hold.
\begin{itemize}
\item[(a)] Suppose that $i\mapsto \frac{\alpha_i}{i}$ is non-decreasing. Then the FDR is non-increasing in each $p_j$, $j \in I_1$. 
\item[(b)] Let $i\mapsto \frac{\alpha_i}{i}$ be non-increasing. Then the FDR is non-decreasing in each $p_j$, $j \in I_1$. 
\end{itemize}
\end{theorem}
\begin{proof}
The combination of (\ref{deffdr}) and (\ref{firsteqfdr}) implies
\begin{align*}
\text{FDR}=\sum\limits_{i\in I_0}\E\left[\frac{\alpha_{R(p^{(i)})}}{R(p^{(i)})}\right].
\end{align*}
Note that whenever $p_j, \ j\in I_1,$ increases then $R(p^{(i)})$ is non-increasing for each $i\in I_0.$ The different requirements for $\frac{\alpha_i}{i}$ establish the results. 
\end{proof}
The BH Theorem is no longer true for arbitrary dependent p-values. \cite{h_j_15.1} derived an elementary example, see Example 3.1, for $m_0=m=2$ of dependent p-values given by a bivariate normal experiment. 
For $\alpha = \frac{1}{2}$ it has been shown that $FDR=\frac{7}{16}<\alpha$ for a specifically chosen positive correlation of the unterlying normal random variables and $FDR=\frac{9}{16}>\alpha$ for some negative correlation, respectively.   
\newline
There are several proposals for multiple tests which allow FDR control (i.e. $\text{FDR}\leqslant \alpha$) under general dependence models. The probably oldest one is the Bonferroni multiple test with constant critical values $\alpha_j=\frac{\alpha}{m}$, which rejects $H_i$ whenever $p_i\leqslant \frac{\alpha}{m}$ for the corresponding p-value.
\begin{example}[Sharpness of the Bonferroni bound]
For each $m=m_0$ there exist p-values with $\text{FDR}=\alpha$ for the Bonferroni multiple test. Consider any copula, i.e. a $m-$dimensional random variable $(U_1,...,U_m)$ with uniformly distributed marginals on $[0,1]$. We define uniformly on $(\frac{i-1}{m},\frac{i}{m})$ distributed variables by $U_i'=\frac{i-1+U_i}{m}$. Let $\sigma$ denote a uniformly distributed permutation of $1,...,m,$ independent of the $U'$s. Then
\begin{align}\label{pvalbon}
(p_1,...,p_m)=(U'_{\sigma(1)},...,U'_{\sigma(m)})
\end{align}
is a vector of p-values with uniformly on $[0,1]$ distributed marginals, which may represent $m=m_0$ true null. The $\text{FDR}$ of the Bonferroni test is given by 
 \begin{align*}
 &\text{FDR}=\text{FWER}=P\left(p_{1:m}\leqslant \frac{\alpha}{m}\right)=P\left(\bigcup\limits_{k=1}^{m}\left\{p_i\leqslant \frac{\alpha}{m} \right\}\right)\\
 &=\sum\limits_{k=1}^{m}P\left(\frac{U_k}{m}<\frac{\alpha}{m},\sigma(k)=1\right)=\sum\limits_{k=1}^{m}P(U_k\leqslant \alpha)P(\sigma(k)=1)=\alpha.
 \end{align*} 

\end{example}
\begin{remark}
The FDR bound $\alpha$ of the Bonferroni test is sharp for different correlated p-values. For $m=2$ the p-values (\ref{pvalbon}) may be positive or negative correlated since $\text{corr}(p_1,p_2)=\text{corr}(U_1,U_2)$ holds.
\end{remark}
Benjamini and Yekutieli (2001) proved that the SU test using critical values 
\begin{align}\label{boundby}
\alpha_i=i\alpha'/\sum\limits_{k=1}^{m}\frac{m}{k}
\end{align}
controls the FDR by $\alpha'$ under general dependence assumptions. If we put $\alpha:=\alpha'/\sum\limits_{k=1}^{m}\frac{1}{k}$ this result equivalently implies the FDR bound
\begin{align}\label{boundbh}
\text{FDR}\leqslant \min (\alpha\sum\limits_{k=1}^{m}\frac{1}{k},1)
\end{align}
for the BH test and sufficiently small $\alpha$ under dependence. \newline
The next lemma implies that the upper bound (\ref{boundbh}) is sharp for $m=2$ and $\alpha<\frac{2}{3}.$

\begin{lemma}
Consider a SU test with critical values $0<\alpha_1< \alpha_2< 1$ for $m=2.$ Under arbitrary dependence of on $[0,1]$ uniformly distributed $p$-values $p_1$ and $p_2$ we have
$\text{FDR}\leqslant\min(\alpha_1+\alpha_2,1).$ If $\alpha_1+\alpha_2<1$ holds the inequality is sharp.
\end{lemma}

\begin{proof}
In case $m_0=1$ the inequality $\text{FDR}\leqslant \alpha_2$  holds obviously. \newline
Consider next $m_0=2.$ If $J_1, \ J_2$ are two intervals the probability 
\begin{align*}
P(p_1\in J_1, \ p_2\in J_2)\leqslant \min(\lebesgue(J_1),\lebesgue(J_2))
\end{align*}
is bounded by the minimal length of the intervals. Consider
\begin{align}\label{threeprob}
\text{FDR}= P(A_1)+P(A_2)+P(C)
\end{align}
based on the events
$A_1=\{p_1\leqslant \alpha_1\}$, $A_2=\{p_2\leqslant \alpha_1,p_1>\alpha_1\}$, $C=\{\alpha_1<p_i\leqslant\alpha_2, \ i=1,2\}.$
Thus,
$
\text{FDR}\leqslant 2\alpha_1 + (\alpha_2-\alpha_1)
$
proves the inequality.\newline
If $\alpha_1+\alpha_2<1$ holds, the following joint distribution of $(p_1,p_2)$ for $m_0=2$ yields the equality $\text{FDR}=\alpha_1+\alpha_2.$
Throughout let $Q_J$ denote the uniform distribution on a subinterval $J\subset [0,1].$ 
Let $p_1$ be any uniformly on $[0,1]$ distributed random variable. For each $x\in[0,1]$ the conditional distribution $P(p_2 \in \cdot \ | \  p_1=x)$ of $p_2$ given $p_1=x$ is specified as follows by uniform distributions on subintervals. Put $P(p_2 \in \cdot \ | \  p_1=x)$ equal to
\begin{align*}
Q_{(1-\alpha_1,1]} & \text{ \ if \ }    0\leqslant x \leqslant \alpha_1, \\
Q_{(\alpha_1,\alpha_2]} &\text{\ if \ } \alpha_1\leqslant x \leqslant \alpha_2,\\
Q_{(0,\alpha_1]\cup (\alpha_2,1-\alpha_1]} & \text{ \ if \ }    \alpha_2\leqslant x \leqslant 1.
\end{align*}
It is easy to see that the marginal distribution  of $p_2$ is uniform and FDR$=\alpha_1+\alpha_2$ by (\ref{threeprob}).
\end{proof}

{\bf Conclusion:} Due to the sharp FDR bounds for the Bonferroni test, the BH and Benjamini and Yekutieli SU tests, respectively,  further dependence assumptions are needed if the upper bounds have to be improved for step up tests.

\section{Results for step down (SD) multiple tests.}
Up to some modifications our method of proof applies to BH-type step down (SD) multiple tests. In contrast to SU tests, see (\ref{defSU}), the number of rejections is given by
\begin{align}\label{defSD}
R=\max\{j: \ p_{i:m}\leqslant \alpha_i \text{ \ for \ all \ }i\leqslant j\}
\end{align}
with $R=0$ if the present set is empty. The SD tests rejects all null hypotheses which belong to p-values $p_{i:m}$ for $i\leqslant R$. Let $t\mapsto \hat{F}_m(t)$ be the empirical distribution function of $p_1,...,p_m.$ Observe first that in case of SU tests the maximum (\ref{defSU}) only relies on so called inspection points $j\in\{ m\hat {F}_m(p_{i:m}): i=1,...,m\}$. This is not true for SD tests in general if tied p-values are present. However, the choice $\alpha_i=a_i$ defined below with new data dependent critical values restricts the comparison (\ref{defSD}) to the inspection points given by the indices $m\hat {F}_m(p_{i:m})$ for ties. In addition to the BH critical values $b_i$ in (\ref{BHcv}) introduce
\begin{align}\label{BHcvSD}
a_i:=b_{m\hat{F}_m(p_{i:m})} \text{ \ for \ } 1\leqslant i\leqslant m.
\end{align}

\begin{theorem} [\textit{SD-Theorem}]
Consider the SD test with $\left(\alpha_i\right)_{i=1,...,m}=\left(a_i\right)_{i=1,...,m}$ or $\left(\alpha_i\right)_{i=1,...,m}=\left(b_i\right)_{i=1,...,m}.$ Under BI assumptions we have the inequality
\begin{align*}
\text{FDR}\leqslant \alpha\frac{m_0}{m}.
\end{align*}
\end{theorem}
\begin{proof}
Analogously to the proof of Theorem \ref{deffdr} (\cite{b_h_95}), equation (\ref{firsteqfdr}) and Fubini's Theorem provide
\begin{align}
\label{SDx}\E\left[\frac{\1(p_i\leqslant \alpha_{R(p)})}{\max(R(p),1)}\right]&=\E\left[\frac{\1(p_i\leqslant \alpha_{R(p)})\1(p_i\leqslant \alpha_{R(p^{(i)})})}{R(p^{(i)})}\right]\\
\label{SD2}&=\E\left[\frac{\1(p_i\leqslant \alpha_{R(p^{(i)})})}{R(p^{(i)})}\right]-\E\left[\frac{\1(p_i\in(\alpha_{R(p)},\alpha_{R(p^{(i)})}])}{R(p^{(i)})}\right]\\
\label{SD3}&\leqslant \E\left[\frac{\1(p_i\leqslant \alpha_{R(p^{(i)})})}{R(p^{(i)})}\right]=\frac{\alpha}{m},
\end{align}
where the second term of the right side of (\ref{SD2}) may be positive only for SD tests. Note that for $i\in I_0$ the p-values $p_i$ only contribute with probability zero to possible ties.
The multiplicity $m_0$ of (\ref{SDx}), due to (\ref{deffdr}), implies the result. 
\end{proof}

\begin{remark}
The presented bound for the FDR of the Benjamini and Hochberg SD test is sharp. Consider $p=(U,0,...,0), \ m_0=1, \ I_0=\{1\}$ and let $U$ be uniformly distributed on $[0,1]$. Then, for both, SU and SD, Benjamini Hochberg tests we have  
\begin{align*}
\text{FDR}=\frac{P(U\leqslant \alpha)}{m}=\frac{\alpha}{m}=\frac{m_0 \alpha}{m}.
\end{align*}
\end{remark}
In contrast to SU tests the first coefficient $b_1$ can be slightly improved without loss of the FDR control for SD test under BI models (see also \cite{me} for more general results). Let $\alpha_0$ denote the positive solution of the equation $(1-\alpha)=\exp(-2\alpha), \ \alpha_0 \approx 0.797.$ Consider now $0<\alpha\leqslant \alpha_0.$ Introduce new critical values
\begin{align*}
c_1=1-(1-\alpha)^{\frac{1}{m}} \ \text{ \ and \ }c_i=b_i \text{ \ for \ } 2\leqslant i\leqslant m.
\end{align*}
It is easy to see that $\frac{\alpha}{m}<c_1\leqslant\frac{2\alpha}{m}$ holds.
\begin{lemma}
Consider the SD test with either $(\alpha_i)_{i=1,...,m}=(c_i)_{i=1,...,m}$ or $(\alpha_i)_{i=1,...,m}=(c_{m\hat {F}_m(p_{i:m})})_{i=1,...,m}.$ Under the BI model the following bound is valid
\begin{align*}
\text{FDR}\leqslant 1-(1-\alpha)^{\frac{m_0}{m}}\leqslant \alpha.
\end{align*}
\end{lemma}
\begin{proof}
Recall the elementary inequality $1-(1-\alpha)^{\beta}> \beta\alpha$ for $0<\beta< 1.$ Within the BI asssumption we may condition under the portion $(p_i)_{i\in I_1}$ which is first assumed to be fixed  if $m_1>0.$ Define $f_1=\min(p_i: i\in I_1)$. In the following we consider two cases.
\begin{itemize}
\item[1.] Assume first that $f_1\leqslant c_1$. Then $\frac{V}{R}>0$ implies $R \geqslant 2.$ If we now substitute $p_i $ by $0, \ i\in I_0,$ letting all others p-values unchanged we are in the position of the SD-Theorem and $\text{FDR}\leqslant\frac{m_0}{m}\alpha$.
\item[2.] For the case $f_1>c_1$ we have
$
\text{FDR}\leqslant P(\min\limits_{i\in I_0} p_i\leqslant c_1)\leqslant 1-(1-c_1)^{m_0},
$
which yields the upper bound $1-(1-\alpha)^{\frac{m_0}{m}}.$
Finally, we integrate the conditional FDR with respect to $(p_i)_{i\in I_1}.$
Our inequalities establish the result since
$
\frac{m_0}{m}\alpha\leqslant 1-(1-\alpha)^{\frac{m_0}{m}}\leqslant \alpha.
$
\end{itemize}
The case $m_0=m$ is trivial.
\end{proof}
Note that the inequality $\text{FDR}=1-(1-\alpha)^{\frac{m_0}{m}}>\frac{m_0}{m}\alpha$ occurs for $m_0<m,$ if $p_i>c_m$ holds for all $i\in I_1.$\newline
In contrast to SU tests the FDR is not monotone for SD tests in the sense of Theorem \ref{beyek}.
\begin{example}[FDR is not monotone for SD]
Consider two i.i.d uniformly distributed on $[0,1]$ random variables $p_1$ and $U.$ Choose $m_0=m_1=1$ and $I_0=\{1\}.$ Define the vector of p-values $(p_1,p_2)$ for four cases $p_2\in\{0, \ \alpha U, \ U, \ (1-\alpha)U+\alpha\}.$ Obviously, $0\leqslant \alpha U\leqslant U\leqslant (1-\alpha)U+\alpha $ is valid. For $p_2=0$ or $(1-\alpha)U+\alpha$ the SD and SU tests with BH critical values coincide and we have $FDR=\frac{\alpha}{2}$ for SD tests. However, for the case $p_2=\alpha U$ and $U$ we obtain 
\begin{align*}
\text{FDR}=\frac{\alpha}{2} -\frac{1}{2}P(\{\frac{\alpha}{2}< p_1\leqslant \alpha,\frac{\alpha}{2}\leqslant p_2< \alpha\}),
\end{align*}
due to (\ref{SD2}), which yields $\text{FDR}(p_2=\alpha U)=\frac{3}{8}\alpha$ and $\text{FDR}(p_2=U)=\frac{\alpha}{2}-\frac{\alpha^2}{8}.$
Thus the FDR for SD test is not monotone in $p_2$ since
\begin{align*}
\text{FDR}(p_2=\alpha U)<\text{FDR}(p_2=U)<\text{FDR}(p_2=(1-\alpha)U+\alpha)
\end{align*}
and $\text{FDR}(p_2=0)>\text{FDR}(p_2=\alpha U)$
hold for SD test with BH critical values $\left(b_i\right)_{i=1,...,n}$.
\end{example}

\bibliographystyle{Chicago}

\end{document}